\documentclass[12pt]{amsart}

\usepackage{enumerate}
\usepackage[pdftex]{graphicx}
\usepackage{amsmath,amsthm,amssymb,latexsym,amscd}
\usepackage{setspace, bbding}
\usepackage{mathrsfs,amsfonts,wasysym,textcomp}
\usepackage{array}
\usepackage[dvipsnames]{xcolor}

\newcommand{\R}{\mathbb{R}}

\newcommand{\diam}{\mbox{Diam}}

\def\H{{\mathbb H}}

\def\H{{\mathbb{H}}}

\newtheorem{thm}{Theorem}[section]
\newtheorem{lem}[thm]{Lemma}
\newtheorem{coro}[thm]{Corollary}
\newtheorem{propo}[thm]{Proposition}

\theoremstyle{definition}		

\newtheorem{ex}[thm]{Example}
\newtheorem{rem}[thm]{Remark}

\numberwithin{equation}{section}

\renewcommand{\epsilon}{\varepsilon}

\usepackage{epstopdf}

\begin{document}

\title[Hyperbolic Projections]{Uniqueness Results for Bodies of Constant Width in the Hyperbolic Plane}
\author{M. Angeles Alfonseca, Michelle Cordier, Dan I. Florentin}

\begin{abstract}
Following Santal\'o's approach, we prove several characterizations of
a disc among bodies of constant width, constant projections lengths,
or constant section lengths on given families of geodesics.
\end{abstract}

\maketitle

\section{Introduction}
The reconstruction of a convex body from quantitative lower dimensional information (such as the areas of sections or projections) is one of the main problems in geometric tomography. One of the main results in this area is Aleksandrov's Theorem. He proved that an origin-symmetric convex body in $\mathbb{R}^n$ is uniquely determined by the $(n-1)$-dimensional volumes of its projections (see \cite[Theorem 3.3.6]{Gardner2}). The focus of the current note is the study of reconstruction problems in the hyperbolic plane $\mathbb{H}^2$.

The question of uniqueness plays a central role in the study of reconstruction, {\it i.e.}, do there exist two bodies $K$ and $L$ with the same projections (or sections) on a given family of geodesic lines? Often, it is assumed that one of the bodies is a disc, and the problem is stated as: 

Does there exist a body $K$ with constant projections or sections on a given family of geodesics, which is not a disc?

The width of a smooth convex body in the hyperbolic plane is defined to be its projection lengths on the family of normal lines (see the precise definition in Section 2). An important difference between the geometry in the Euclidean and the hyperbolic plane is that, while the concepts of width o f a body and its projections lengths coincide in the Euclidean setting, they capture different information in the hyperbolic one. For example, any Euclidean planar body of constant width has constant projection lengths on every line, but this no longer holds in the hyperbolic plane, see Remark \ref{rem_min-proj-arbitrary}. 

Bodies of constant width in hyperbolic spaces have been studied  in the literature, and several analogues of Euclidean results have been established (see e.g. \cite{Ara}, \cite{Dek}, \cite{GRST}, \cite{Jeronimo}, \cite{Leicht}, \cite{Sant} and references therein). Problems about sections of convex bodies in hyperbolic spaces (related to the Busemann-Petty problem) have been studied in \cite{HY}, \cite{Y}. In this paper we study  bodies of constant width by their sections or projections on all lines through a given point. In particular, we consider the following four classes of convex bodies in $\mathbb{H}^2$:
{\sl
\begin{enumerate}[$(i)$]
    \item origin-symmetric bodies.
    \item Constant width bodies.
    \item Bodies of constant projection lengths on all lines through a point.
    \item Bodies of constant sections lengths on all lines through a point.
\end{enumerate}
}
We show that the disc in $\mathbb{H}^2$ is uniquely characterized by any two of the above properties. Some of these characterizations follow immediately from the definitions (e.g. origin symmetry plus constant section lengths), but are included for completeness. Our main results are:

\bigskip
{\bf Theorem 4.1.} {\it 
Let $K\subset\mathbb{H}^2$ be a $C^1$-smooth convex body. If $K$ is origin-symmetric and has constant width, then $K$ is a disc.}
\bigskip

{\bf Theorem 4.4.} {\it 
Let $K\subset\mathbb{H}^2$ be a $C^1$-smooth convex body, containing the origin in its interior. If $K$ is a body of constant width and all projections of $K$ on the geodesics passing through the origin have constant length, then $K$ is a disc.}
\bigskip

Theorem \ref{Thm_cw+H0=disk} is a result that has no Euclidean analog, since in $\mathbb{E}^2$ having constant width and constant projection lengths is the same property. A version of Theorem \ref{width-sym} was proven in \cite{Jeronimo},
where the authors use a different definition for width than the one we present below.
\bigskip

The paper is organized as follows. In Section \ref{defs} we introduce the needed definitions and notation. In Section \ref{counterex} we collect several auxiliary lemmas that will be needed for the main results. In Section \ref{unique} we prove Theorems 4.1 and 4.4. We also include a unique reconstruction result for general convex bodies from projection lengths on two families of geodesics. Finally, the Appendix contains some computations for the hyperbolic Reuleaux triangle.

\bigskip

{\it Acknowledgments:} We would like to thank Dmitry Ryabogin and Vlad Yaskin for many fruitful discussions about this paper, in particular for the result in Section 4.2.

\section{Definitions and Notation}\label{defs}

Throughout the paper, we work interchangeably in the Poincar\'e disc model and the Poincar\'e upper half-plane model of the hyperbolic plane $\mathbb{H}^2$. We refer to the book \cite[Volume 2, Chapter 19]{B} for definitions and properties of the hyperbolic plane and its various models. Here we recall some basic facts: Given two points $x,y \in \mathbb{H}^2$, there is a unique geodesic line joining them. Also, both models are conformal, and hence angles and the notion of perpendicularity are defined as in the Euclidean setting. In particular, given a point $x$ on a geodesic $\Gamma$ and an angle $\alpha$  there exists a unique geodesic passing through $x$ and forming an angle $\alpha$ with $\Gamma$. We use the disc model when studying centrally symmetric bodies, while the upper half-plane model is well suited to the geometry of the proof of Lemma \ref{maxproj}. In only one instance (Lemma \ref{HB}) we use the Beltrami-Klein disc model, in which geodesics are straight lines; this allows us to directly use the corresponding Euclidean property. For more detail on the Beltrami-Klein disc model see \cite{R}.

A body $K \subseteq \mathbb{H}^2$ is a compact set that is equal to the closure of its non-empty interior. A body $K$ is called convex if its boundary cannot be intersected by a geodesic in more than two points (with the exception that an arc of a geodesic may be part of the boundary). Any geodesic that has only a common point with $\partial K$ or with a complete arc of $\partial K$ is called a {\it supporting geodesic} of $K$. We will often assume that the boundary of $K$ is $C^1$-smooth and contains no geodesic segments.

Given two points $x,y \in \mathbb{H}^2$, we denote by $[x,y]$ the unique geodesic segment joining them, and by $\big|[x,y]\big|$ the length of this segment. Given a geodesic $\Gamma$ in $\mathbb{H}^2$ and a point $x\in \mathbb{H}^2$, we denote by $n_\Gamma(x)$ the perpendicular (or normal) geodesic to $\Gamma$ passing through $x$. In particular, if $\partial K$ is $C^1$-smooth, given $x\in \partial K$ and letting $\Gamma$ be a supporting geodesic of $K$ at $x$, we will write $n_K(x)$ (instead of $n_\Gamma(x)$) for the normal geodesic to $\partial K$ at the point $x$. The set  $\mathcal{N}_K=\{n_K(x): x\in \partial K\}$ of all normal geodesics to $\partial K$ will be called the normal field of $K$. A geodesic $\Gamma \in \mathcal{N}_K$ is called a {\it double normal} geodesic of $K$ if it intersects $\partial K$ perpendicularly at two points. 

The {\it projection} of the body $K$ onto a geodesic $\Gamma$ is defined as 
\[	
    P_\Gamma(K) = \{x\in \Gamma \,|\, K\cap n_\Gamma(x) \neq \emptyset \}.
\]
Several non-equivalent definitions of width have been introduced in previous works  (see \cite{Sant} and \cite{Jeronimo}). We follow Santal\'o's definition \cite{Sant}. Given $x\in \partial K$ and the normal geodesic $n_K(x)$, let $\Gamma_2$ be another supporting geodesic of $K$ which is also orthogonal to $n_K(x)$. If $\phi(x)$ is the point of intersection of $n_K(x)$ and $\Gamma_2$, we define the width of $K$ at $x$ by $w_K(x)=\big|[x,\phi(x)]\big|$  (see Figure \ref{width}, left). We note that the width of $K$ at $x$ equals the length of the projection  $P_{n_K(x)}(K)$, but it need not equal the length of the projection of $K$  on another geodesic perpendicular to $\Gamma_2$ (such as the geodesic passing through $x'$ in Figure \ref{width}, right). Thus, we see that in contrast with the Euclidean situation, there is no natural way of defining the support function of a convex body in the hyperbolic setting.

\begin{figure}[h]
    \centering
    \includegraphics[height=.8in]{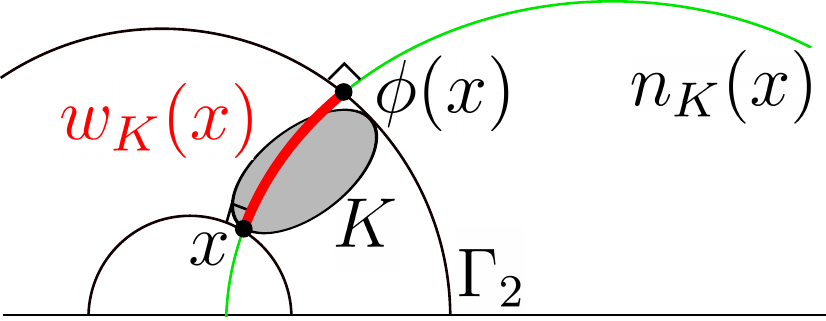} $~~~~~$
    \includegraphics[height=.8in]{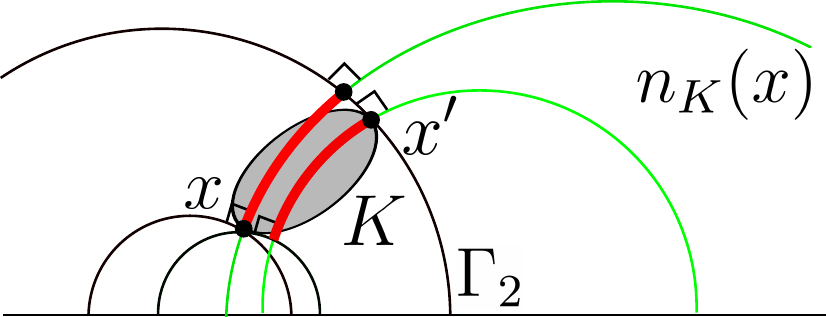}
    \caption{The widths $w_K(x)$ and $w_K(x')$.}
    \label{width}
\end{figure}

With this definition, bodies of constant width are, like in the Euclidean case, precisely those for which every normal geodesic is double normal ({\it i.e.,} for each $x\in \partial K$, the corresponding point $\phi(x)$ is also on $\partial K$; see \cite[pg. 410]{Sant} for details). A body of constant width $K$ in $\mathbb{E}^2$ has constant projection lengths on every line, but in $\mathbb{H}^2$ this is not the case; in fact one can always find a geodesic on which $K$ has an arbitrarily small projection length. 

A point $p$ is said to be equichordal to a body $K$ if all intersections of $K$ by lines through $p$ have equal lengths. The set of all lines passing through a fixed point $p\in \mathbb{H}^2$ will be denoted by $H_p$. 
Given a geodesic $\Gamma$, we denote by $\Gamma^\perp$ the set of all the geodesics perpendicular to $\Gamma$.

Let $K$ be a convex body in the Poincar\'e disc model of $\mathbb{H}^2$, containing the origin $O$ in its interior. We say that $K$ is origin-symmetric if for every $x\in K$, we have $-x \in K$. 

\section{Auxiliary Lemmas}\label{counterex}
In this section we state some basic metric facts regarding convex bodies
and their orthogonal projections in the hyperbolic plane. The first one is the characterization of a hyperbolic convex body as the intersection of all half-planes containing it. We prove this using the Beltrami-Klein disc model, where the geodesics are the chords of the (open) disc $D$, and thus the compact convex sets in $D\cong\H^2$ correspond to the compact convex subsets of the open disc in $\R^2$.

 \begin{lem}
 \label{HB}
 Let $K\subset \H^2$ be a convex body. Then $K$
 equals the intersection of all the half-planes containing it.
 \end{lem}
 
 \begin{proof}
 In the Beltrami-Klein disc model, for any supporting half-plane $H$ we denote by $\Tilde{H}$ the (unique) half-plane of $\R^2$ such that $H=\Tilde{H}\cap D$.
 \[
 \bigcap_{K\subset H} H= 
 \bigcap_{K\subset H} (\Tilde{H}\cap D)=
 \left(\bigcap_{K\subset H} \Tilde{H}\right)\cap D =
 K\cap D = K,
 \]
 where the third equality holds since in the Euclidean plane, a convex sets equals the intersection of all containing half-planes.
 \end{proof}

\begin{rem}\label{symmetry}
The projection of a centrally symmetric convex body $K$ is symmetric, hence knowing the lengths of the projections on  geodesic from $H_0$ determines the half-planes containing $K$. Without central symmetry the lengths of projections on $H_0$ do not determine the body, as seen in the following example.
\end{rem}

\begin{ex}\label{ears-remark}
We present two non-congruent polygons $P_1,P_2$, that have equal lengths of projections on $H_0$. We will follow the construction in the Euclidean plane from \cite[Theorem 3.3.15]{Gardner2}. In Figure \ref{ears}, the polygons $P_1$ and $P_2$ are presented in the upper half-plane, and the origin from the  Beltrami-Klein disc is mapped to the point $(0,1)$ on the upper half-plane.

\begin{figure}[h]
    \centering
    \includegraphics[height=2in]{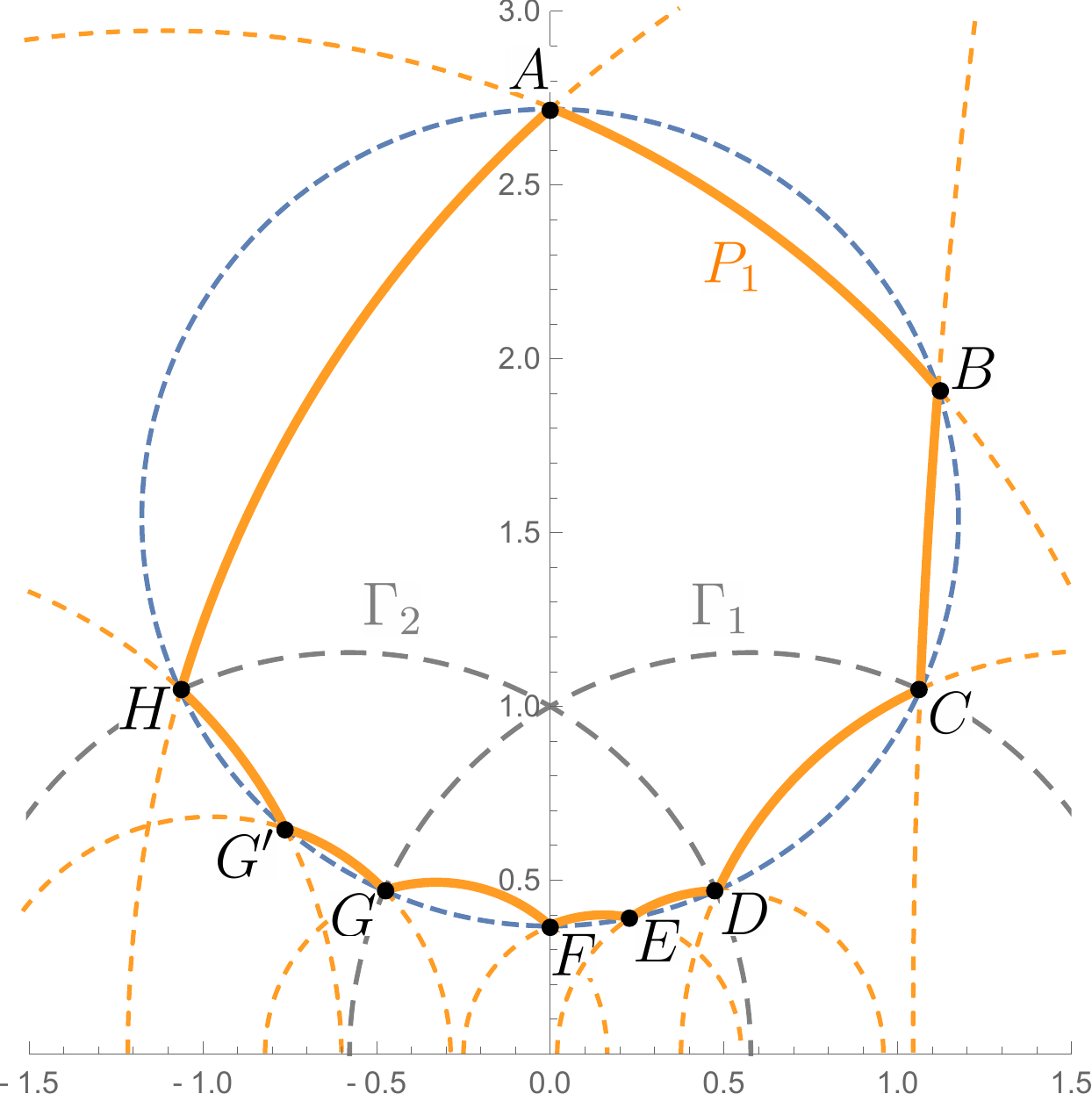}
    \includegraphics[height=2in]{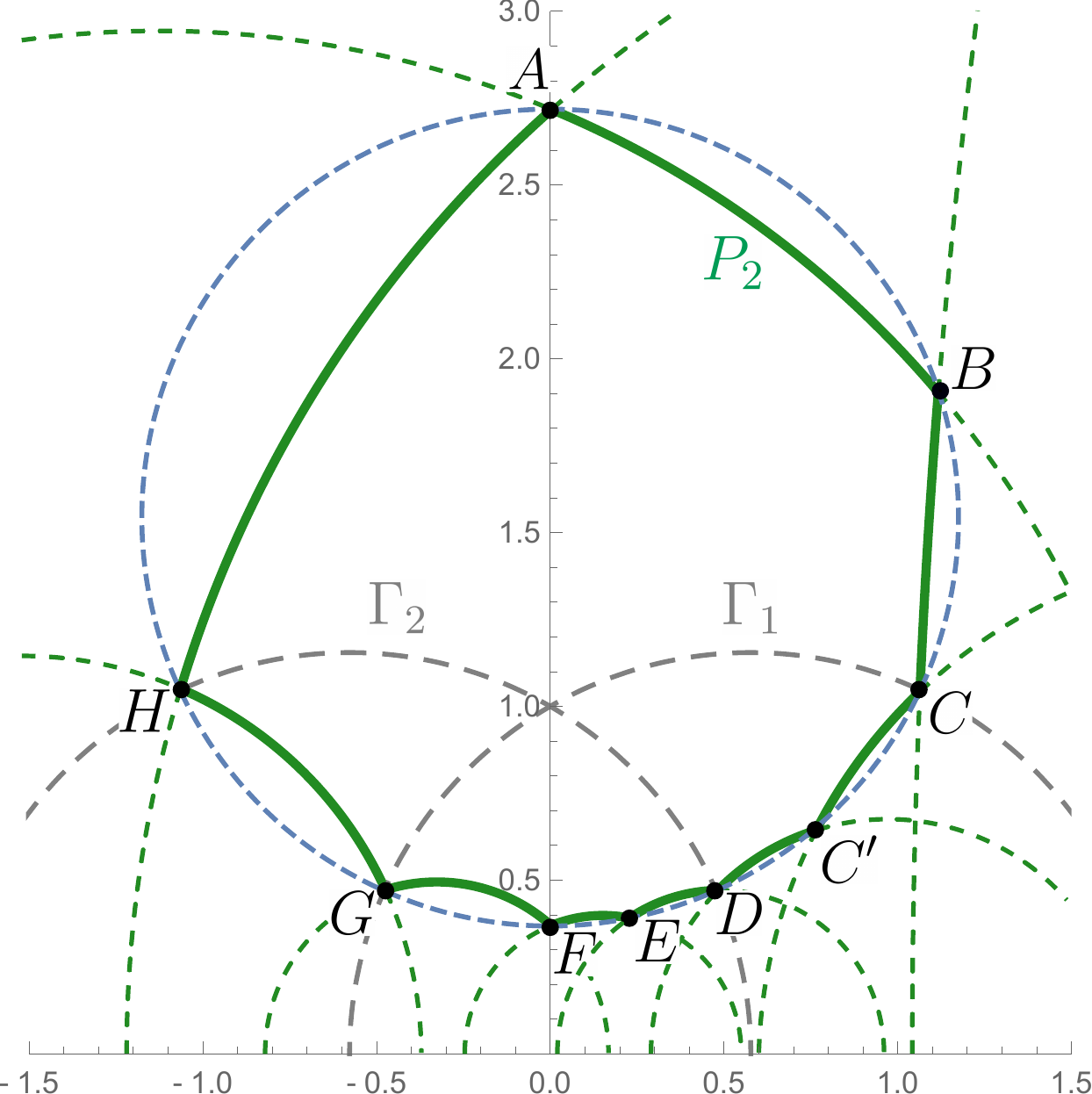}
    \caption{Two non-congruent bodies with equal lengths of projections on $H_{(0,1)}$.}
    \label{ears}
\end{figure}

The lengths of the projections of $P_1$ and $P_2$  on each geodesic in $H_{(0,1)}$ coincide. Indeed, let $\Gamma_1 \in H_{(0,1)}$ be the geodesic passing through $C$ and $G$, and $\Gamma_2 \in H_{(0,1)}$ the one passing through $D$ and $H$. The two bodies are reflections of each other in the region between $\Gamma_1$ and $\Gamma_2$, and identical outside this region. Moreover, the parts of $P_1$ or $P_2$ that lie between $\Gamma_1$ and $\Gamma_2$ do not affect any projections of the bodies on geodesics lying outside of this region, while their projections on geodesics from $H_{(0,1)}$ that lie inside this region have equal lengths.
\end{ex}

Gardner's construction provides a counterexample where both bodies (with equal lengths of projections) are polygons. One may also construct pairs of smooth bodies with the same property, as seen in the following example.

\begin{ex}\label{perturb-proj} Using Lemma \ref{HB}, we can construct a (non-symmetric) body whose projections on $H_0$ have constant length. In the Poincar\'e disc, given the unit ball $B$ centered at $O$, and a  smooth, odd function $f:S^1\to\R$, for each line $L_u$ through $O$ in the direction $u \in S^1$  we translate the symmetric segment $L_u \cap B$ along $L_u$ by $\epsilon f(u)$ units obtaining a segment $I_u$. Let $\Gamma^1_u,\Gamma^2_u$ be the two geodesics perpendicular to $L_u$ that pass through each of the two endpoints of the segment $I_u$, and denote by $S_u$ the region between $\Gamma^1_u$ and $\Gamma^2_u$ (which we will call a ``slab"). Then, we have that  $P_{L_u}(S_u)=I_u$, and we define $K$ to be the intersection of these slabs:
\[
K = \bigcap_{u\in S^1} S_u.
\]
If $\epsilon$ is small enough, by Lemma \ref{HB} the projections of $K$ on $L_u$ are equal to $I_u$. Therefore, $K$ has projections of constant length on $H_0$ and it is not a disc. 

\bigskip   
   
\noindent We remark that, similarly, the body $\widetilde{K}$ whose radial function is $\rho_{\widetilde{K}}(u)=c+\epsilon f(u)$, for some constant $c<1$ is a body with sections of constant length on $H_0$ which is not a disc. 
\end{ex}
   
As we mentioned in Section \ref{defs}, bodies of constant width are characterized by the fact that all their normal geodesics are double normals. In the next Lemma we observe that  on a double normal geodesic of $K$, projections and sections of $K$ coincide.
   
\begin{lem} 
   \label{binorm} 
   Let $K$ be a convex body and let $\Gamma$ be a geodesic intersecting $\partial K$ at two points. Then $\Gamma$ is a double normal geodesic for $\partial K$ if and only if 
   $P_\Gamma(K)=K\cap \Gamma$.
    \end{lem}
 
\begin{figure}[h]
    \centering
    \includegraphics[height=1.25in]{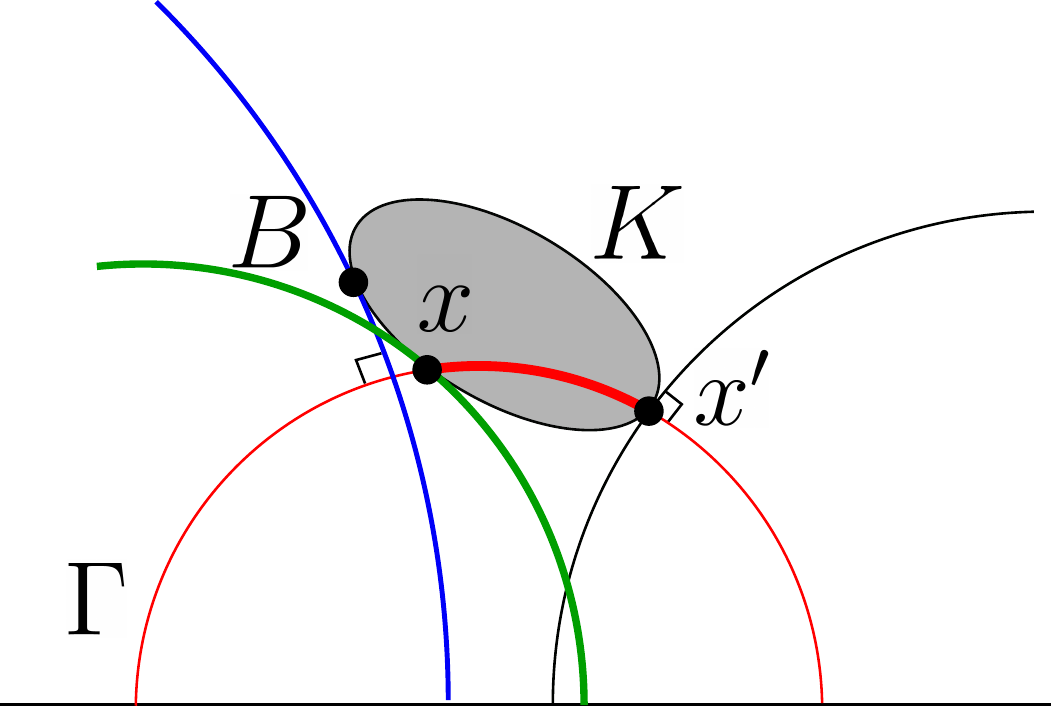}
    \caption{$K \cap \Gamma \subsetneq P_{\Gamma}(K)$.}
    \label{proj_sec_backward2}
\end{figure}

  \begin{proof}
  Denote by $x,x'$ the points of intersection of $\Gamma$ and $\partial K$. If $\Gamma$ is a double normal, the geodesics perpendicular to $\Gamma$ at $x$ and at $x'$ support $K$, and hence $P_\Gamma(K) = [x,x']=K\cap \Gamma$. Conversely, if $\Gamma$ is not normal to $\partial K$ at $x$, then $K$ is not contained in any of the two half-planes with boundary perpendicular to $\Gamma$ at $x$ (see Figure \ref{proj_sec_backward2}). Thus $x$ is an interior point of the segment $P_\Gamma(K)$, and clearly a boundary point of $[x,x']=K\cap \Gamma$, so we get $ K\cap \Gamma \subsetneq P_\Gamma(K) $.
  \end{proof}

\begin{rem}\label{rem-diam-is-width}
Every $C^1$-smooth convex body $K$ has at least one
double normal. Indeed, the diameter of $K$, being the
maximal distance between two points in $K$, is
attained on a segment which is perpendicular to
$\partial K$ at both ends. In fact, the next lemma states that, using our definition of width, the diameter coincides with the maximal width of $K$.
\end{rem}

\begin{lem}\label{diameterwidth}
Let $K$ be a $C^1$-smooth convex body. Then
\[
\max_{x\in \partial K} \left\{w_K(x)\right\}
=
\diam(K),
\]
where $\diam(K) := \max_{x,y\in \partial K} \left\{\big|[x,y]\big|\right\}$
\end{lem}

\begin{proof} Denoting $W=
\max_{x\in \partial K} \left\{w_K(x)\right\}$, we get by Remark \ref{rem-diam-is-width} that
$Diam(K)\le W$. On the other hand, let
$x \in \partial K$ be the maximum point of
$w_K:\partial K\to \R^+$, that is,
$W=w_K(x)=|[x,\phi(x)]|$. Let $\Gamma_2$ be the
supporting geodesic of $K$ passing through $\phi(x)$,
and let $P\in \partial K$ be its point of tangency
to $K$. In the (possibly degenerate) right triangle
$\Delta x \phi(x) P$, the hypotenuse $xP$ is the
longest edge. Hence, 
\[
    W = |[x,\phi(x)]| \le |[x,P]| \le Diam(K).
\]
\end{proof}

\begin{lem}
\label{normfield}
The normal field to a $C^1$-smooth convex body $K$  covers the interior of $K$.
\end{lem}

\begin{proof}
Let $x$ be an interior point of $K$. If
$y\in \partial K$ is the closest boundary point to $x$, then the segment $[x,y]$ is perpendicular to $\partial K$ at $y$ (since $y$ is a critical point of the function $d:\partial K \to \R^+$ given by $d(\cdot)=|[x,\cdot]|$). Equivalently; $x\in n_K(y)$.
\end{proof}

Next we present two Lemmas of Santal\'o which are
used in the proofs of Theorems \ref{width-sym} and
\ref{Thm_cw+H0=disk}.
The first lemma states that two normal lines of a
constant width body must intersect. For the reader's
convenience we include Santal\'o's proof here.

\begin{lem}\label{Lem-Sant2-double normals-intersect}{(Santal\'o, \cite[pg. 411]{Sant})}
Let $K\subset \H^2$ be a $C^1$-smooth convex body of constant width. Then any two lines normal to $\partial K$ intersect at an interior point of $K$.
\end{lem}

\begin{figure}[h]
    \centering
    \includegraphics[height=1in]{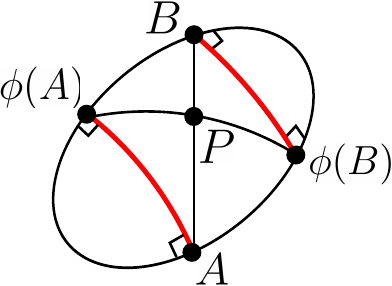}
    \caption{Lemma \ref{Lem-Sant2-double normals-intersect} (Santal\'o)}
    \label{sant2}
\end{figure}

\begin{proof}
Let $A, B\in \partial K$ be two distinct points. 
Suppose, for the sake of contradiction, that the
normal line segments $[A,\phi(A)],\, [B,\phi(B)]$ do
not intersect in the interior of $K$. Then the four
points $A,B,\phi(A), \phi(B)$ form a non-degenerate
quadrilateral $Q$, which (excluding the vertices) is
contained in the interior of $K$, as seen in Figure
\ref{sant2} (up to relabeling
$A\leftrightarrow\phi(A)$, or
$B\leftrightarrow\phi(B)$).
Now let $p\in K$ be the point of intersection of the
diagonals $[A,B]$ and $[\phi(A),\phi(B)]$ of $Q$.
From the triangle inequality it follows that
\[
\big|[A, \phi(A)]\big| <
\big|[A,       p]\big| +
\big|[p, \phi(A)]\big|,
\]
\[
\big|[B, \phi(B)]\big| <
\big|[B,       p]\big| +
\big|[p, \phi(B)]\big|.
\]
By Lemma \ref{diameterwidth}, the diameter of
$K$ is given by its width, so we get
\begin{eqnarray*}
2 Diam(K)
&=& w_K(A) + w_K(B) = \\
&=&\big|[A, \phi(A)]\big| + \big|[B, \phi(B)]\big|<\\
&<&\big|[A,       p]\big| + \big|[p, \phi(A)]\big|
 + \big|[B,       p]\big| + \big|[p, \phi(B)]\big|=\\
&=&\big|[A,    B]\big| + \big|[\phi(A),\phi(B)]\big|
\le 2 Diam(K).
\end{eqnarray*}
From this contradiction, we conclude that the
normal line segments $n_K(A),\,n_K(B)$ must
intersect at an interior point of $K$.
\end{proof}

We also use a second lemma of Santal\'o's, which we state without a proof. 

\begin{lem}\label{santalolemma}{(Santal\'o, \cite[pg. 406]{Sant})}
Let $K$ be a convex body with $C^1$-smooth boundary curve $\partial K$. Let $\Gamma$ be a geodesic that intersects $\partial K$ at the points $M,N$ and that is normal to $\partial K$ at both points. Then on the arc of $\partial K$ joining $M$ and $N$ there is a unique point $x_0$ such that $n_K(x_0)$ is normal to $\Gamma$ (see Figure \ref{santalo}). In fact, the angle $\alpha(x)$ which the normal geodesics to $\Gamma$ form with $\partial K$ is (strictly) increasing, with $\alpha(M)=0$ and $\alpha(N)=\pi$.
\end{lem}
\begin{figure}[h]
    \centering
    \includegraphics[height=2in]{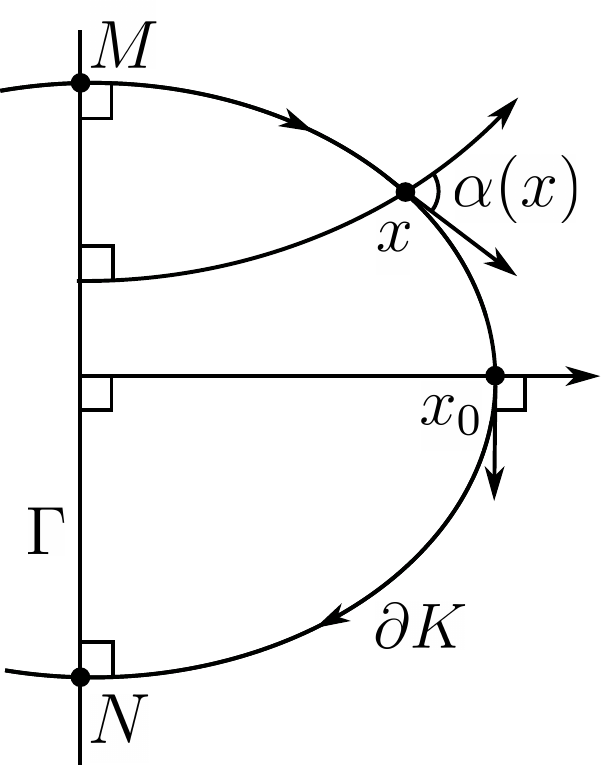}
    \caption{Lemma \ref{santalolemma} (Santal\'o)}
    \label{santalo}
\end{figure}

\section{Results on unique reconstruction from projection lengths}\label{unique}

\subsection{Characterizations of the disc}
In $\mathbb{E}^2$, a body of constant width has constant projection lengths on all lines. In $\mathbb{H}^2$ this is not the case, since there are geodesics on which the projection length is arbitrarily small. In addition, we show in the Appendix that a hyperbolic Reuleaux triangle constructed from a circle centered at $O$ has constant width, but does not have constant projection lengths on $H_0$. 

A disc centered at the origin has the following four properties:
\begin{enumerate}[$(i)$]
    \item Origin-symmetry.
    \item Constant width. 
    \item Constant projection lengths on $H_0$. 
    \item Constant sections lengths on $H_0$.
\end{enumerate}
In this section we prove that any two of these properties characterize the disc.

The characterization of the disc by $(i)$ and $(iii)$ was already observed in Remark \ref{symmetry}. Similarly, the characterization by $(i)$ and $(iv)$ follows directly from the definitions.

Our main results are the characterizations by $(i)$ and $(ii)$, and by $(ii)$ and $(iii)$. In Theorem \ref{width-sym}, we show that the disc is the only $C^1$-smooth origin-symmetric body of constant width. In Theorem \ref{Thm_cw+H0=disk} we prove that the disc is the only $C^1$-smooth body of constant width with constant projection lengths on $H_0$. The two remaining characterizations follow as Corollary \ref{coro1} and Theorem \ref{coro2}.

\begin{thm}\label{width-sym}
If $K$ is an origin-symmetric convex body of constant width $w$ and
$\partial K$ is a $C^1$-smooth curve, then $K$ is a disc.
\end{thm}

\begin{figure}[h]
    \centering
    \includegraphics[height=1.8in]{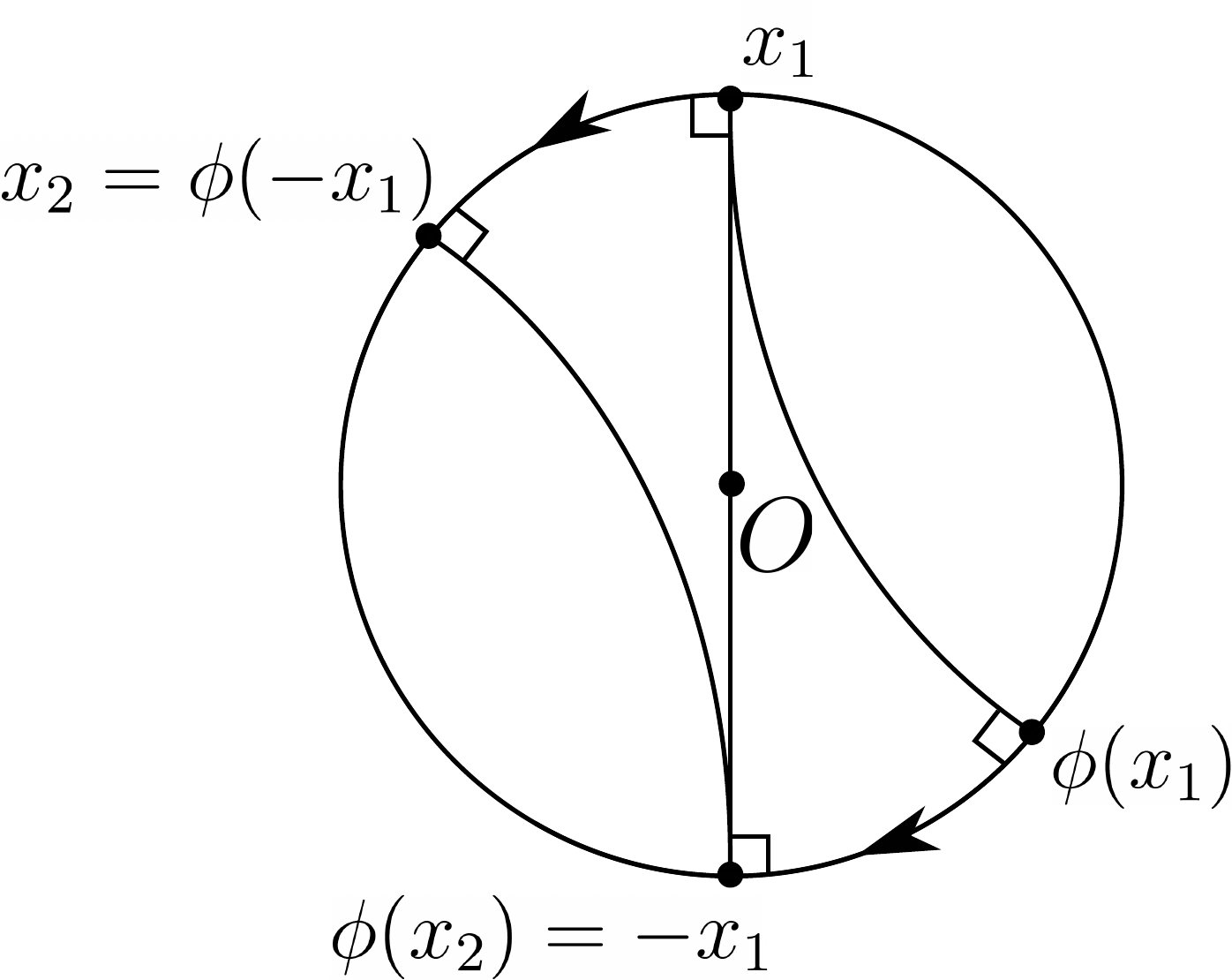}
    \caption{Proof of Theorem \ref{width-sym}.}
    \label{mon_inc}
\end{figure}

\begin{proof}
Let $x\in \partial K$, and let $n_K(x)$ be the normal geodesic to $K$ at $x$. Denote by $\phi(x)$ the other point where $n_K(x)$ intersects the boundary of $K$. Since $K$ has constant width, the geodesics  $n_K(\phi(x))$ and $n_K(x)$ are double normals and thus they coincide. 
Therefore, the function defined on $\partial K$ by $x \rightarrow \phi(x)$ is an involution ({\it i.e.,} $\phi(\phi(x))=x$). 

Suppose that for a point $x_1 \in \partial K$,  we have that $\phi(x_1) \not= -x_1$. Denote $x_2=\phi(-x_1)$. Then, by origin symmetry, the geodesic segments $n_K(x_1) \cap K $ and $n_K(-x_1) \cap K$ do not intersect, see Figure \ref{mon_inc}. But this contradicts Lemma \ref{Lem-Sant2-double normals-intersect}, and thus  
 we must have $\phi(x_1)=-x_1$. Since the point $x_1$ was arbitrary, all the normal lines to $\partial K$ are radii, and for every $x\in \partial K$,  $\big|[x,-x]\big|$ is equal to the width $w$. This means that $K$ is a disc.
\end{proof}

A version of Theorem \ref{width-sym} was proven in \cite{Jeronimo}, where the definition of width considered is different than Santal\'o's (however, both definitions coincide on the class of bodies of constant width). Our proof is also different, based on Santal\'o's Lemma \ref{Lem-Sant2-double normals-intersect} regarding the intersection of any two normals of a body of constant width.

\bigskip 

Next, we prove that the maximal projection length of a constant width body equals its width, and is attained only on the normal field of the body. This will be the main ingredient in the proof of Theorem \ref{Thm_cw+H0=disk}.

\begin{propo}\label{maxproj}
Let $K \subseteq \mathbb{H}^2$ be a convex body of constant width $w$, and let $M$ be a geodesic line.
Then
\[
|P_M(K)| \le w,
\]
and equality holds if and only if $M$ is a double normal
geodesic of $K$.
\end{propo}

\begin{figure}[h]
    \centering
    \includegraphics[height=1.8in]{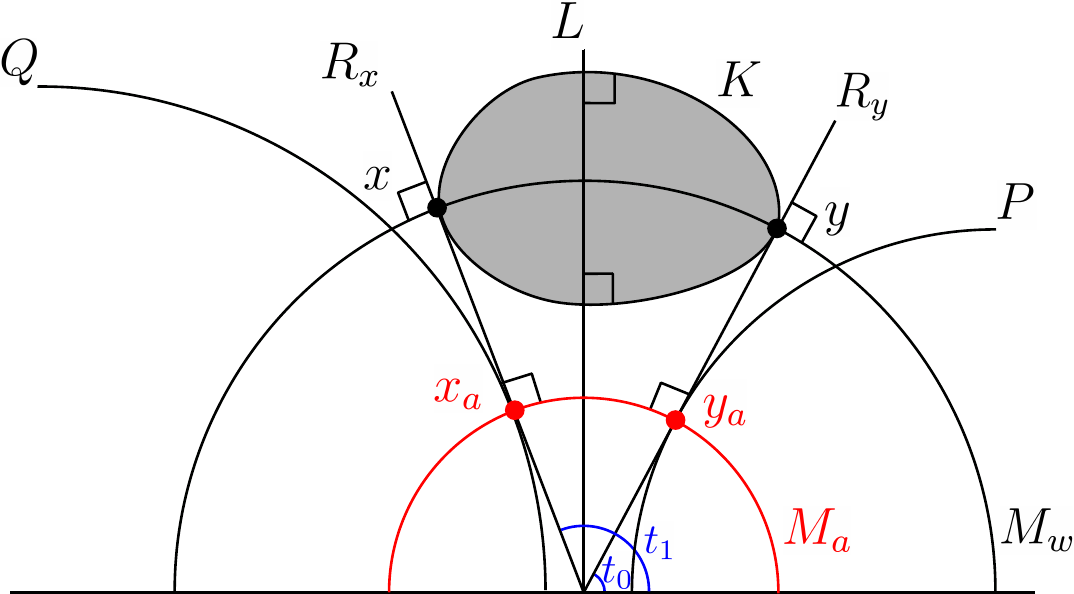}
    \caption{Proof of Proposition \ref{maxproj}.}
    \label{cone}
\end{figure}

\begin{proof}
Given a geodesic line $M$, there exists a (unique) geodesic line $L$
which is perpendicular to both $\partial K$ and $M$. We shall
therefore prove the following (only formally stronger) statement.
Given a binormal geodesic $L$ of $K$, and $M\in L^\perp$, we have
\[
|P_M(K)| \le w,
\]
and equality holds if and only if $M$ is the unique binormal of $K$
which is perpendicular to $L$.

We consider $K$ in the upper half-plane model of $\mathbb{H}^2$. Let
$L$ be a normal geodesic to $K$. By applying an isometry, we may
assume that $L$ is the $y$-axis, so that the family $L^\perp$ of
geodesics perpendicular to $L$ consists of concentric half circles
(see Figure \ref{cone}). By Santal\'o's Lemma \ref{santalolemma},
only one line of $L^\perp$ (which we denote by $M_w$) intersects
$\partial K$ perpendicularly. Since $K$ has constant width, the
projection of $K$ on the normal line $M_w$ has length $w$.

Let $x,y$ be the points of intersection between $M_w$ and
$\partial K$, and let $R_x$, $R_y$ be the rays emanating from the
base point of $L$ and passing through $x,y$, respectively. Each such
ray has the property of being equidistant to $L$, {\it i.e.} it
consists of points at some fixed distance from $L$. Equivalently,
they are orthogonal to every line of $L^\perp$. Therefore every
geodesic $M_a\in L^\perp$, intersects $R_x,R_y$ orthogonally at the
points $x_a,y_a$, and we have $|[x_a,y_a]| = |[x,y]| = w$.

Since $\{x\} = K\cap R_x$, the geodesics $Q,P$ perpendicular to $M_a$
at $x_a,y_a$ do not intersect $K$ for $a\neq w$, see Figure
\ref{cone}. Thus, the projection of $K$ on $M_a$ is strictly
contained in the segment $[x_a,y_a]$, and we have
\[
P_{M_a}(K) <
\big|[x_a,y_a]\big| =
\big|[x,y]\big| =
P_{M_w}(K) =
w.
\]
\end{proof}

\begin{rem}\label{rem_min-proj-arbitrary}
In contrast to the maximal projection being obtained on a diameter
(see Lemma \ref{diameterwidth}), a minimal projection does not exist.
Indeed, for any convex body $K\subset\H^2$, an arbitrarily small
projection can be found. It can be shown that if a line $L$
intersects $K$, and $x\in L$ is a point sufficiently far from $K$,
then the projection of $K$ onto the line $M$ perpendicular to $L$ at
$x$ can be made arbitrarily small. Thus the spectrum of projection
lengths is the interval $(0,Diam(K)]$.
\end{rem}

\begin{thm}
\label{Thm_cw+H0=disk}
Let $K$ be a body of constant width $w$, containing the origin in its interior, and such that $\partial K$ is a $C^1$-smooth curve. If all projection lengths of $K$ on $H_0$ are constant, then $K$ is a disc.
\end{thm}

\begin{proof}
By Lemma \ref{normfield}, the normal field of $K$ covers the interior
of $K$. Therefore there is a normal geodesic to $\partial K$ which
belongs to $H_0$. On this geodesic, the width and the length of the
projection coincide. Hence, all projections of $K$ on $H_0$ have
length $w$. By Proposition \ref{maxproj}, this implies that each line
in $H_0$ is a normal line, {\it i.e.}, $\mathcal{N}_K = H_0$, which
implies that $K$ is a disc.

\end{proof}

\begin{coro}\label{coro1}
Let $K$ be a body of constant width $w$, containing the origin in its interior, and such that $\partial K$ is a $C^1$-smooth curve. If all sections of $K$ on $H_0$ have  constant length, then $K$ is a disc.
\end{coro}

\begin{proof}
The argument is very similar to the proof  of Theorem \ref{Thm_cw+H0=disk}. Since the normal field of $K$ covers the interior, there is a double normal geodesic $\Gamma$ that belongs to $H_0$. By Lemma \ref{binorm} and Proposition \ref{maxproj}, the length of the section $K\cap \Gamma$ equals the width $w$. Hence, all sections of $K$ on $H_0$ have length $w$, and the conclusion follows again as in the proof of Theorem \ref{Thm_cw+H0=disk}. 
\end{proof}

Finally, to prove the characterization of the disc by properties
$(iii)$ and $(iv)$, we require the following lemma.

\begin{lem}\label{equichordal}
Let $K$ be a convex body with a $C^1$-smooth boundary. Let $p \in
int(K)$ be an equichordal point for $K$. Then a normal geodesic line
passing through $p$ must be a double normal.
\end{lem}

\begin{proof}
Let $x\in \partial K$ such that $p\in n_K(x)$. For each
$y \in \partial K$, let $\Gamma_y$ be the geodesic passing through
$y$ and $p$, and let $y'$ the other boundary point of $K$ belonging
to $\Gamma_y$. We consider a $C^1$ parametrization $r(t)$,
$t\in [0,2\pi)$ of $\partial K$, such that $r(0)=x$, and with the
property that if $r(t)=y$ for $t\in [0,\pi)$, then $r(t+\pi)=y'$.

We define the function $g:[0, 2\pi)\to ]\R^+$ to be the distance from
the equichordal point $p$ to the corresponding boundary point of $K$,
i.e. $g(t) = |[p,r(t)]|$. The equichordal hypothesis says that for
some constant $c$, any $t\in [0, 2\pi)$ satisfies
$g(t) + g(t+\pi) = c$. Since the normal geodesic to $K$ at $x$ passes
through $p$, we have that $g'(0)=0$. Then, for small $\epsilon$,
\[
   g(\epsilon)=g(0)+ O(\epsilon^2),
\]
and 
\[
   g(\pi +\epsilon)=g(\pi)+ \epsilon g'(\pi) +O(\epsilon^2).
\]
Adding the previous two equations we obtain that for every $\epsilon$ small enough
\[
  c=c+\epsilon g'(\pi)+ O(\epsilon^2)
\]
which implies that $g'(\pi)=0$, thus $n_K(x)$ is a double normal.
\end{proof}

\begin{thm}\label{coro2}
Let $K$ be a convex body, containing the origin in its interior, such that $\partial K$ is a $C^1$-smooth curve. If  every section of $K$ by a geodesic from $H_0$ has constant length $c_1$, and every projection of $K$ on a geodesic in $H_0$ has constant length $c_2$, then $K$ is a disc (and $c_1=c_2$).
\end{thm}

\begin{proof}
Let $x\in\partial K$ be the closest point in $\partial K$ to the origin. Then the line $L$ containing the segment $[0,x]$ is a normal line of $K$. By Lemma \ref{equichordal}, $L$ is a double normal line, hence $P_L(K)=K\cap L$ by Lemma \ref{binorm}. This implies $c_1=c_2$. But then all projections and sections on $H_0$ have the same length, and therefore every geodesic in $H_0$ is a double normal geodesic of $K$. Since $\partial K$ is a curve orthogonal to all rays through the origin, $K$ is a disc.
\end{proof}

\begin{figure}[ht]
	\centering
	\includegraphics[scale=1]{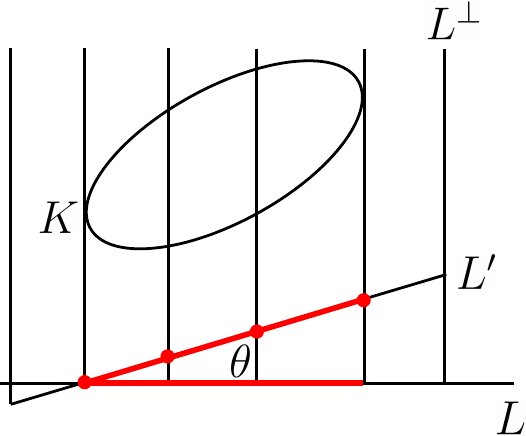}
	\caption{Orthogonal and non-orthogonal projections in Euclidean setting.}
	\label{nonorth}
\end{figure}

\subsection{Orthogonal and non-orthogonal projections}	
As we have seen, the lengths of projections and the width are two distinct sources of information in the hyperbolic setting. Theorem \ref{Thm_cw+H0=disk} is based on using both to obtain a hyperbolic result which has no Euclidean analogue. Another example of this situation is given by orthogonal and non-orthogonal projections: In the Euclidean plane, knowing the length of an orthogonal projection of a convex body $K$ on a line $L$ gives us the same information as knowing the length of a non-orthogonal projection on a line $L'$ forming a known  angle $\theta$ with $L$ (see Figure \ref{nonorth}). However, in the hyperbolic setting, knowing the lengths of orthogonal and non-orthogonal projections provides additional information that allows us to recover the body. 

Let $K$ be a convex body and let $p$ be an interior point of $K$.  For each line $L \in H_p$ we consider another line $L'$, chosen in the following way. We can assume that $L$ is a vertical line in the upper half-plane model. If $C_1,C_2$ are the two supporting geodesics of $K$ that are perpendicular to $L$, we choose  another vertical line $L'$ at (Euclidean) distance $t>0$ from $L$, such that $L'$ intersects both supporting geodesics (see Figure \ref{hypnonorth}). By compactness, the positive number $t$ can be chosen as the same for all lines $L \in H_p$. 

\begin{propo}
Let $K$ be a convex body and let $p$ be an interior point of $K$. Then $K$ is determined by the  projection lengths
\[
  \{ |P_L(K)|,L\in H_p \} \cup  \{ |P_{L'}(K)|,L\in H_p \}.
\]
\end{propo}

\begin{proof}
\begin{figure}[ht]
	\centering
	\includegraphics[scale=1]{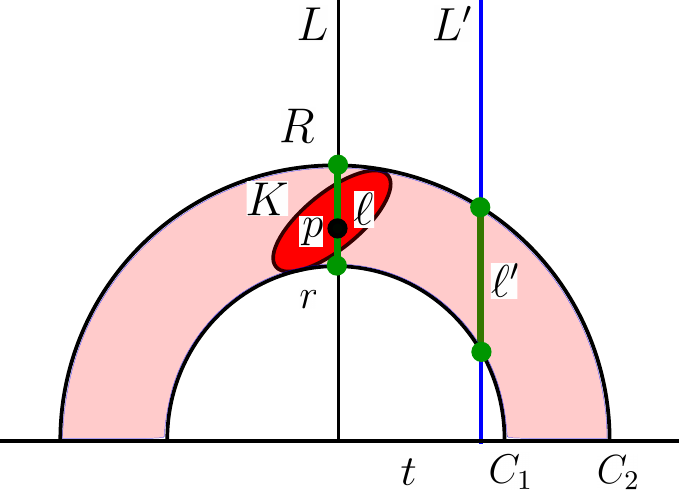}
	\caption{Orthogonal and non-orthogonal projections in the hyperbolic half-space.}
	\label{hypnonorth}
\end{figure}

For a given geodesic $L \in H_p$, we will call $\ell=|P_L(K)|$ and $\ell'=|P_{L'}(K)|$. Let $R$ be the radius of $C_1$ and $r$ be the radius of $C_2$ (see Figure  \ref{hypnonorth}). Then
\[
\ell=\ln \left(\frac{R}{r} \right)  \;\;\;\; {\mbox{ and }}  \ell'=\ln \left(\frac{ \sqrt{R^2-t^2}}{\sqrt{r^2-t^2}} \right).
\]
Therefore, $R$ and $r$ are determined from $\ell$, $\ell'$ and $t$, and the intersection of all such strips between supporting geodesics uniquely determines the body.
\end{proof}

\bigskip

\section{Appendix}
We show that the hyperbolic Reuleaux triangle defined below has constant width and non constant sections and projections on $H_0$. We work in the upper-plane model of $\mathbb{H}^2$. On the circle with hyperbolic center $(0,1)$ and hyperbolic radius 1, we choose three equidistant points, with one of them being $(0,e)$ (see Figure \ref{reuleaux}). For each of the three points, we trace the hyperbolic circle centered at that point and containing the other two. The resulting figure $T$ is the hyperbolic Reuleaux triangle, which has constant width since given one of the circular sides of the triangle, all of its normals pass through the opposite vertex of the triangle. The width can be computed by finding the distance between the intersection points of the Reuleaux triangle with the $y$-axis, and its value is $1.78774413\ldots$. 
\begin{figure}[h]
    \centering
    \includegraphics[height=2in]{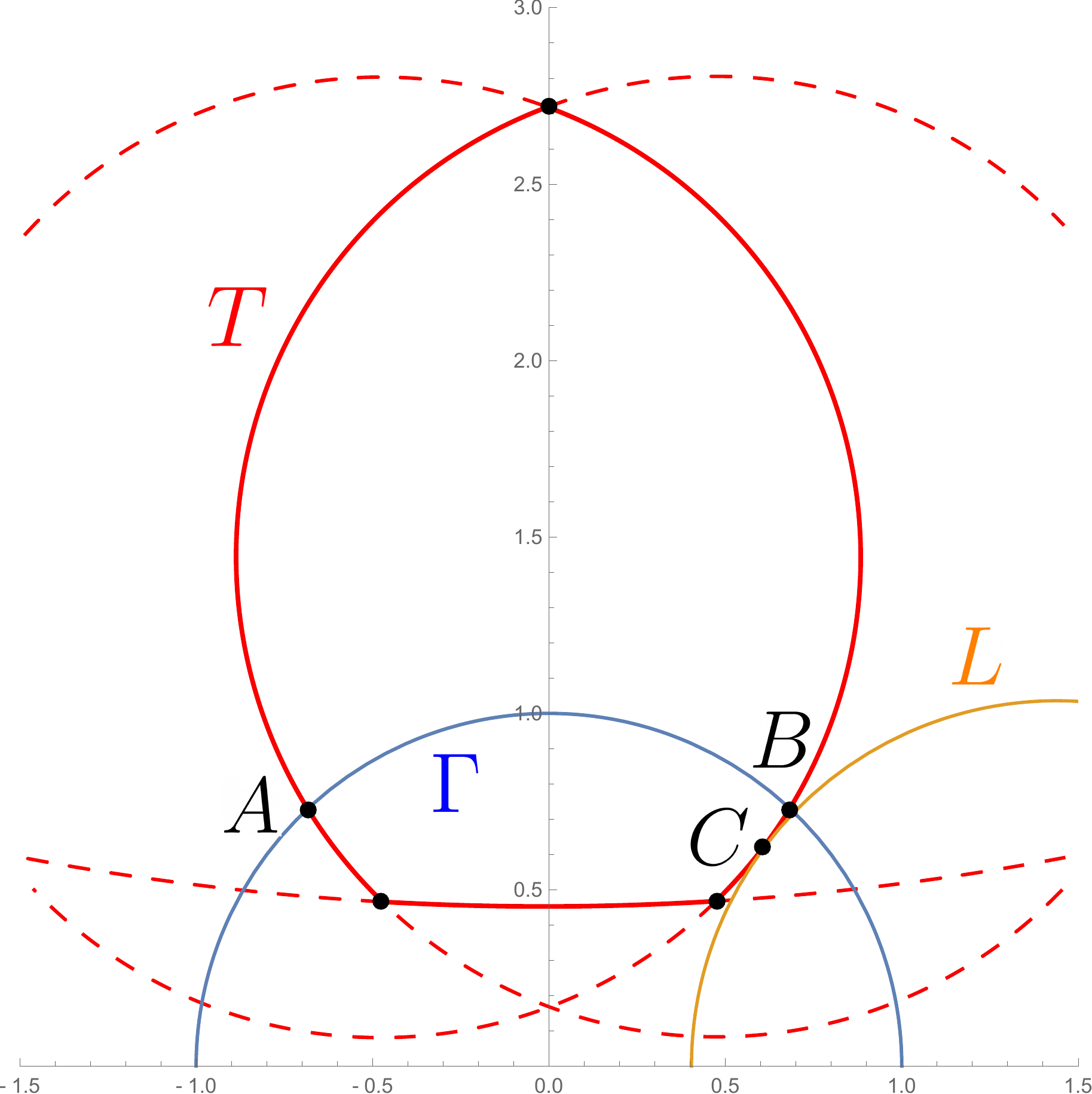}
    \caption{Reuleaux triangle}
    \label{reuleaux}
\end{figure}

We show that there are two geodesics in $H_0$ on which the sections and projections of $T$ are different. First, on the vertical line passing through $(0,1)$, the section and the projection of the Reuleaux triangle coincide, and their length is equal to the width of $T$. Consider now the geodesic $\Gamma$ through $(0,1)$ forming a 90 degree angle with the vertical geodesic. We explicitly determine the points $A$ and $B$ where $\Gamma$ intersects the boundary of $T$ by solving the system formed by the equations of these two circles. The length of the section of $T$ by $\Gamma$ can then be exactly computed with Mathematica (though the expression is very long), and its approximate value is $1.67461854\dots$. Similarly, to compute the length of the projection of $T$ on this geodesic, we first find the geodesic $L$, normal to $\Gamma$, that is tangent to $T$  (the point of tangency is $C$ in Figure \ref{reuleaux}); next, we explicitly find the point of intersection of $L$ and $\Gamma$, again by solving the system of the equations of the circles. The length of the projection of $T$ onto $\Gamma$ is approximately $1.7133762\ldots$.

We observe that one can also construct a $C^1$-smooth body of constant width which is a perturbation of a disc and does not have constant sections nor projections through $H_0$.

\end{document}